\documentclass[oneside,reqno]{amsart}

\title[Accelerating the Distributed Kaczmarz]{Accelerating the Distributed Kaczmarz Algorithm by Strong Over-relaxation}
\author{Riley Borgard$^{4}$ \and Steven N. Harding$^{3}$ \and Haley Duba${}^{5}$ \and Chloe Makdad${}^{2}$ \and Jay Mayfield${}^{3}$ \and Randal Tuggle${}^{1}$ \and Eric Weber${}^{3}$ \\ \\ \tiny ${}^{1}$Berry College,\, ${}^{2}$Butler University,\, ${}^{3}$Iowa State University,\, ${}^{4}$Purdue University,\, ${}^{5}$Wheaton College}
\usepackage{amssymb}
\usepackage{amsmath}
\usepackage{mathrsfs}
\usepackage{commath}
\usepackage{graphicx}
\usepackage{esint}
\usepackage{enumitem}
\usepackage{qtree}
\usepackage{pgfplots}
\usepackage{amsthm}
\usepackage{wasysym}
\usepackage{tikz,tkz-euclide}
\usepackage{bm}
\usepackage{verbatim}
\usepackage{subcaption}

\usetkzobj{all}
\usetikzlibrary{shapes,snakes}
\usetikzlibrary{arrows.meta}
\usepackage{lscape}
\newtheoremstyle{case}{}{}{}{}{}{:}{ }{}
\theoremstyle{case}

\theoremstyle{theorem}
\newtheorem{theorem}{Theorem}[section]
\newtheorem{corollary}{Corollary}[theorem]
\newtheorem{lemma}[theorem]{Lemma}
\newtheorem{remark}[theorem]{Remark}
\newtheorem{definition}{Definition}
\newtheorem{proposition}[theorem]{Proposition}

\usepackage{booktabs}

\usepackage{caption}

\begin{document}

\date{\today}

\maketitle

\begin{abstract}

The distributed Kaczmarz algorithm is an adaptation of the standard Kaczmarz algorithm to the situation in which data is distributed throughout a network represented by a tree. We isolate substructures of the network and study convergence of the distributed Kazmarz algorithm for relatively large relaxation parameters associated to these substructures. If the system is consistent, then the algorithm converges to the solution of minimal norm; however, if the system is inconsistent, then the algorithm converges to an approximated least-squares solution that is dependent on the parameters and the network topology. We show that the relaxation parameters may be larger than the standard upper-bound in literature in this context and provide numerical experiments to support our results.

\end{abstract}
\section{Introduction}
The Kaczmarz algorithm, introduced in \cite{Kaczmarz1937}, is a classic row-action projection method for solving a system of linear equations $A\vec{x}=\vec{b}$ where $A$ is a complex-valued $k \times d$ matrix. We denote row $i$ of the matrix $A$ by $\vec{a}_i^*$ so that the corresponding equation in the system is $\langle\vec{x},\vec{a}_i\rangle = \vec{a}_i^*\vec{x} = b_i$. Herein, we provide a self-contained description of the Kaczmarz algorithm for completeness. Given an initial vector $\vec{x}^{(0)}$, we find the orthogonal projection of $\vec{x}^{(0)}$ onto the hyperplane $\vec{a}_1^*\vec{x} = b_1$ to obtain the estimate $\vec{x}^{(1)}$. We repeat this procedure, iterating through the rows of $A$; once we obtain $\vec{x}^{(k)}$, we return to the first equation to obtain $\vec{x}^{(k+1)}$ and continue through the matrix as before. More precisely, for $i = n\pmod{k} + 1$, we have
\begin{equation} \label{eq1}
\vec{x}^{(n+1)} = \vec{x}^{(n)}+ \omega \frac{b_i-\vec{a}_i^*\vec{x}^{(n)}}{\lVert \vec{a}_i\rVert ^2}\vec{a}_i,
\end{equation}
where $\lVert\cdot\rVert$ is the Euclidean norm, and $\omega$ is a relaxation parameter. Stefan Kaczmarz showed in \cite{Kaczmarz1937} that if the system is consistent and the solution is unique, then the sequence $\{\vec{x}^{(n)}\}$ converges to the solution with $\omega = 1$. Later, in \cite{Tanabe1971}, Tanabe showed that the sequence $\{\vec{x}^{(n)}\}$ converges to the solution of minimal norm when the system is consistent for any $\omega \in (0,2)$.  When the system is inconsistent, it was shown in \cite{Eggermont1981} (see also \cite{Natterer2001}) that for every $\omega \in (0,2)$, the sequence $\{\vec{x}^{(n)}\}$ converges, and for $\omega$ small, the limit is an approximation of a weighted least-squares solution.

Since each estimate is obtained by projecting the previous estimate onto the appropriate hyperplane, the Kaczmarz algorithm is well-suited for an adaptation to a network structure where each equation in the system corresponds to a node in a tree, an undirected graph excluding cycles. This was formalized in \cite{Hegde2019}. Such a system is said to be distributed, as any node is uninformed of the equation of another node. A distributed system has many benefits in practical applications, e.g. data that is too large to store on a single server or cannot be explicitly shared for privacy reasons. Further, for large distributed systems, we can exploit parallelism to speed up the real time of iterations within the algorithm.

\subsection{Main Results}
Our main focus in the present paper is to consider an extension of the Kaczmarz algorithm that can solve a system of linear equations when the equations are distributed across a network.  This extension was introduced in \cite{Hegde2019}, where it was shown that the distributed form of the Kaczmarz algorithm converges for any relaxation parameter $\omega \in (0,2)$.  It was also shown that, as is the case with the classical Kaczmarz algorithm, the convergence rate can be accelerated by choosing $\omega > 1$.  Moreover, it was observed that convergence can occur with $\omega > 2$, which cannot happen in the classical case.

Our main results concern proving convergence for relaxation parameters $\omega > 2$, as well as determining what the algorithm converges to.  First, we prove that with large relaxation parameters that satisfy a certain admissibility condition (Definition \ref{D:admissible}), when the system is consistent, the distributed Kaczmarz algorithm converges to the solution of minimal norm independent of the relaxation parameters (Theorem \ref{Theorem3.1}).  Second, we prove that under the same admissibility conditions, when the system is inconsistent, the distributed Kaczmarz algorithm will yield approximations of a weighted least-squares solution as the parameters tend to 0 (Theorem \ref{mainthm}). 

We then consider possible values for the relaxation parameters that satisfy the admissibility condition.  We prove an estimate on the sizes of the relaxation parameters at nodes that are near the leaves of the tree (Corollary \ref{PGnormcor}).  Our estimate allows for relaxation parameters that are larger than 2.  In Section \ref{sec:ex}, we present numerical examples that illustrate convergence with relaxation parameters greater than 2 that is faster than with parameters less than 2.
\subsection{Notation}
We define the network for a distributed system as a tree in graph theory parlance--that is, a connected graph consisting of $k$ vertices, each corresponding to one equation in the system, with edges that connect particular pairs of vertices in such a way that there are no cycles. Herein, we only consider trees which are rooted, having a single vertex $r$ designated as the root. We denote arbitrary vertices of the tree by either $u$ or $v$. We write $u \preceq v$ when either $u = v$ or $u$ is on a path from $r$ to $v$. We further write $u \to v$ or $v \leftarrow u$ when $u \neq v$ and $u \preceq x \preceq v$ implies either $u = x$ or $x = v$. From this partial ordering on the set of vertices, we define a leaf of the tree as a vertex $\ell$ satisfying $\ell \preceq u$ implies $u = \ell$, and we denote the collection of all of the leaves by $\mathcal{L}$. Whenever necessary, we enumerate the leaves as $\ell_1,\ell_2,...,\ell_t$.

A weight $w$ is a positive function on the paths of the tree, which we denote by $w(u,v)$ where $u \preceq v$, that satisfies the following two conditions: For every vertex $u \notin \mathcal{L}$,
\begin{equation}
\label{eq:sumweights1}
\sum_{v\,:\,u \to v}w(u,v) = 1
\end{equation}
and, if $u = u_1 \to u_2 \to \cdots\to u_J = v$, then
\begin{equation}
\label{eq5}
w(u,v) = \prod_{j=1}^{J-1} w(u_j,u_{j+1}).
\end{equation}

When working with a distributed network represented by a rooted tree, it is convenient to index each equation by the corresponding vertex, and we proceed with this convention throughout the remainder of the paper. We recall, for a linear transformation $T$ on $\mathcal{H}$, the kernel (null space) $\mathcal{N}(T) = \{\vec{x}\in\mathcal{H}\,:\,T\vec{x} = \vec{0}\}$ and the range $\mathcal{R}(T) = \{T\vec{x}\,:\,\vec{x}\in\mathcal{H}\}$. We define $S_v\vec{x} := \vec{a}_v^*\vec{x}$, and let $P_v$ be the orthogonal projection onto $\mathcal{N}(S_v)$,
\begin{equation} \label{operatorP}
    P_v\vec{x} = (I - S_v^*(S_vS_v^*)^{-1}S_v)\vec{x} = \vec{x} - \dfrac{\vec{a}_v^*\vec{x}}{\|\vec{a}_v\|^2}\vec{a}_v.
\end{equation}
Then, let $Q_v$ be the affine projection onto the hyperplane $S_v\vec{x} = b_v$,
\begin{equation}\label{operatorQ}
    Q_v\vec{x} = \vec{x} + \frac{b_v - \vec{a}_v^*\vec{x}}{\|\vec{a}_v\|^2}\vec{a}_v.
\end{equation}
The relationship between $P_v$ and $Q_v$ is then
\begin{align}\label{littleh}
Q_v\vec{x} = P_v\vec{x} + \vec{h}_v,
\end{align}
where $\vec{h}_v$ is the vector that satisfies $S_v\vec{h}_v = b_v$ and is orthogonal to $\mathcal{N}(S_v)$. The vector $\vec{\omega}$ refers to the entire collection of relaxation parameters, and notation associated with $\vec{\omega}$ implies a dependence on the relaxation parameters. Specifically, the component $\omega_v$ in $\vec{\omega}$ is the relaxation parameter associated with vertex $v$. We further define the associated operators $P_v^{\vec{\omega}}$ and $Q^{\vec{\omega}}_v$ by
\begin{align} 
    P^{\vec{\omega}}_v &= (1 - \omega_v)I + \omega_v P_v,\label{operatorPw} \\
    Q^{\vec{\omega}}_v &= (1 - \omega_v)I + \omega_vQ_v. \label{operatorQw}
\end{align}
The relationship between $P^{\vec{\omega}}_v$ and $Q^{\vec{\omega}}_v$ is then
\begin{align}
    Q^{\vec{\omega}}_v\vec{x} = P^{\vec{\omega}}_v\vec{x} + \omega_v\vec{h}_v, \label{QandP}
\end{align}
with $\vec{h}_v$ as in Equation \ref{littleh}.

\begin{lemma}\label{contraction}
Let $\omega_v \in (0,2)$. Then $P_v^{\vec{\omega}}$ is a contraction (i.e., $\|P_v^{\vec{\omega}}\| \leq 1$). Moreover, $\|P_v^{\vec{\omega}}\vec{x}\| \leq \|\vec{x}\|$ with equality if and only if $\vec{x} \in \mathcal{N}(S_v)$.
\end{lemma}

\begin{proof}
The argument is fairly straightforward, yet it illustrates the sufficient condition that $\omega_v \in (0,2)$.
\begin{align*}
\|P^{\vec{\omega}}_v\vec{x}\|^2 &= \|P_v(\vec{x}) + (1 - \omega_v)(I - P_v)(\vec{x})\|^2 \\
&= \|P_v\vec{x}\|^2 + |1 - \omega_v|^2\|(I - P_v)(\vec{x})\|^2 \\
&\leq \|P_v\vec{x}\|^2 + \|(I - P_v)(\vec{x})\|^2 = \|\vec{x}\|^2
\end{align*}
with equality if and only if $\vec{x} = P_v\vec{x}$.
\end{proof}


\subsection{The Distributed Kaczmarz Algorithm with Relaxation}
Each iteration of the distributed Kaczmarz algorithm begins with an estimate $\vec{x}^{(n)}$ at the root of the tree; the superscript indicates the number of times that we iterated through the tree to obtain the estimate for some given initial estimate $\vec{x}^{(0)}$. An iteration of the algorithm occurs in two stages: {\em dispersion} followed by {\em pooling}. In the dispersion stage, a new estimate is first calculated at the root using the Kaczmarz update with the relaxation parameter $\omega_r$,
\[
\vec{x}_r^{(n)} = \vec{x}^{(n)} + \omega_r\dfrac{b_r - \vec{a}_r^*\vec{x}^{(n)}}{\|\vec{a}_r\|^2}\vec{a}_r =: Q_r^{\vec{\omega}}\vec{x}^{(n)}.
\]
Each subsequent vertex $v \neq r$ receives an input estimate $\vec{x}^{(n)}_u$ from its parent $u$ (i.e., $u \leftarrow v$), and a new estimate is calculated at the vertex $v$ using the Kaczmarz update with relaxation parameter $\omega_v$,
\[
\vec{x}^{(n)}_v = \vec{x}^{(n)}_u + \omega_v\frac{b_v-\vec{a}_v^*\vec{x}_u^{(n)}}{\lVert \vec{a}_v\rVert ^2}\vec{a}_v =: Q^{\vec{\omega}}_v\vec{x}^{(n)}_u.
\]
Each leaf $\ell$ then has its own estimate $\vec{x}^{(n)}_\ell$ at the end of the dispersion stage.

In the pooling stage, we back-propagate the leaf estimates, weighting along the edges, to obtain the next iterate in the algorithm,
\[
\vec{x}^{(n+1)} = \sum_{\ell \in \mathcal{L}}w(r,\ell)\vec{x}^{(n)}_\ell.
\]

It was shown in \cite{Hegde2019} that the distributed Kaczmarz algorithm with fixed relaxation parameters $\omega_v = \omega \in (0,2)$ converges to the solution of minimal norm when the system is consistent and converges to an approximate solution related to some weighted least-squares solution, dependent on the parameters and the network topology, when the system is inconsistent.


\subsection{Substructures of a Network}\label{sec2}
A \emph{subnetwork} $G$ of a network is a subset of vertices and edges satisfying the following conditions:
\begin{enumerate}
    \item If $u \in G$, $x \leftarrow u$ and $x \rightarrow v$, then $v \in G$.
    \item If $u \in G$ and $u \to v$, then $G$ contains $v$ and the edge between $u$ and $v$.
    \item Let $u,v \in G$. The path from $u$ to $v$ does not include the root.
\end{enumerate}
The topology of a subnetwork can thus be characterized as follows: It is either a network itself or a \emph{leaf subnetwork} (a set containing only leaves) or a combination of both. Figure \ref{fig:stepsize} illustrates a network with both types of subnetworks.

\begin{figure}[h]
\centering
\begin{tikzpicture}    

\def\w{0.8}
\def\u{0.6}

\node[circle,fill=black!20] (1) at (1*\w,0*\u) {1};
\node[circle,fill=black!20] (2) at (-1*\w,-2*\u) {2};
\node[circle,fill=black!20] (3) at (4*\w,-2*\u) {3};
\node[circle,fill=black!20] (4) at (-2*\w,-4*\u) {4};
\node[circle,fill=black!20] (5) at (0*\w,-4*\u) {5};
\node[circle,fill=black!20] (6) at (2*\w,-4*\u) {6};
\node[circle,fill=black!20] (7) at (4*\w,-4*\u) {7};
\node[circle,fill=black!20] (8) at (6*\w,-4*\u) {8};

\draw[-,thick] (1) to (2);
\draw[-,thick] (1) to (3);
\draw[-,thick] (2) to (4);
\draw[-,thick] (2) to (5);
\draw[-,thick] (3) to (6);
\draw[-,thick] (3) to (7);
\draw[-,thick] (3) to (8);

\draw[black,thick,dashed] (-1*\w,-3.3*\u) ellipse (1.7*\w cm and 1.9*\u cm);
\draw[black,thick,dashed] (4*\w,-4*\u) ellipse (2.5*\w cm and 1.1*\u cm);



\end{tikzpicture}
\caption{A network with the two types of subnetworks}
\label{fig:stepsize}
\end{figure}
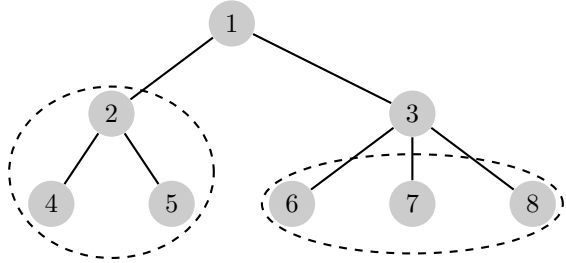

Throughout the paper, we assume that every leaf is included in a subnetwork. The purpose of each subnetwork is to isolate a substructure of the network, so we assume that the subnetworks are pairwise disjoint. We denote the subnetworks by $G_1$, $G_2$, ..., $G_c$ and denote the vertex that immediately precedes $G_i$ by $g_i$. We further denote the leaves in $G_i$ by $\ell_{i,1}$, $\ell_{i,2}$, ...,$\ell_{i,t_i}$. We last denote the root of the largest tree in $G_i$ with the leaf $\ell_{i,j}$ by $r_{i,j}$. For example, in Figure \ref{fig:stepsize}, we have the following:

\begin{itemize}
\item $G_1 = \{2,4,5\}$, $G_2 = \{6,7,8\}$
\item $g_1 = 1$, $g_2 = 3$
\item $\ell_{1,1} = 4$, $\ell_{1,2} = 5$, $\ell_{2,1} = 6$, $\ell_{2,2} = 7$, $\ell_{2,3} = 8$
\item $r_{1,1} = 2$, $r_{1,2} = 2$, $r_{2,1} = 6$, $r_{2,2} = 7$, $r_{2,3} = 8$
\end{itemize}

As each subnetwork is a forest of trees, we may interpret an iteration of $\vec{x}^{(n)}_{g_i}$ through the subnetwork $G_i$ as a weighted average of the iterations through the corresponding trees. We therefore define the following operators:
\begin{align}
P_{G_i}^{\vec{\omega}} &=\sum_{j=1}^{t_i}w(g_i,\ell_{i,j})P_{\ell_{i,j}}^{\vec{\omega}}...P_{r_{i,j}}^{\vec{\omega}},
\label{eq15}\\
P^{\vec{\omega}}_{G_i,r} &=P^{\vec{\omega}}_{G_i} P^{\vec{\omega}}_{g_i}...P^{\vec{\omega}}_r, \label{eq16}\\
P^{\vec{\omega}}&=\sum_{i=1}^c w(r,g_i)P_{G_i,r}^{\vec{\omega}}, \label{Pomega}
\end{align}
where $P^{\vec{\omega}}_v...P^{\vec{\omega}}_u$ with $u \preceq v$ is the composition of those operators $P^{\vec{\omega}}_x$ where $u \preceq x \preceq v$ in the appropriate order designated by the path from $u$ to $v$. We define analogous operators in $Q$.

We will show in Section \ref{inconsistent} that the substructures in a network generally admit relatively large relaxation parameters for convergence. We assume that the relaxation parameters satisfy certain admissibility conditions.

\begin{definition} \label{D:admissible}
We say that the relaxation parameters $\omega_{v}$ are \emph{admissible} provided that:
\begin{enumerate}
    \item If $v\not\in G_i$ for every $i$, then $\omega_v\in (0,2)$.
    \item For each $i \in \{1,2,...,c\}$, there exists a constant $\alpha_i < 1$ such that
    \[
    \lVert P_{G_i}^{\vec{\omega}}\vec{x}\rVert\le \alpha_i \lVert \vec{x}\rVert
    \]
    for all $\vec{x
    }\in\mathrm{span}\{\vec{a}_u : u\in G_i\}$.
\end{enumerate}
\end{definition}

\begin{lemma}\label{contractions}
If the relaxation parameters are admissible, then $P_{G_i}^{\vec{\omega}}$, $P_{G_i,r}^{\vec{\omega}}$ and $P^{\vec{\omega}}$ are contractions.
\end{lemma}

\begin{proof}
Suppose that $\vec{x} \in \{\vec{a}_u\,:\,u \in G_i\}^\perp$, the subspace orthogonal to the vectors in the set $\{\vec{a}_u\,:\, u \in G_i\}$. Then $P^{\vec{\omega}}_u\vec{x} = \vec{x}$ for every $u \in G_i$, and we have
\[
P^{\vec{\omega}}_{G_i}\vec{x} = \sum_{j=1}^{t_i}w(g_i,\ell_{i,j})P^{\vec{\omega}}_{\ell_{i,j}}\cdots P^{\vec{\omega}}_{r_{i,j}}\vec{x}= \vec{x}.
\]
Since $\text{span}\{\vec{a}_u\,:\,u \in G_i\}$ is an invariant subspace for $P^{\vec{\omega}}_{G_i}$, the operator $P^{\vec{\omega}}_{G_i}$ is a contraction. Then, from Lemma \ref{contraction}, it follows that $P^{\vec{\omega}}_{G_i,r}$ and, subsequently, $P^{\vec{\omega}}$ are contractions.
\end{proof}
\subsection{Related Work}


The Kaczmarz method was originally introduced in \cite{Kaczmarz1937}. Variations on the Kaczmarz method allowed for relaxation parameters \cite{Tanabe1971}, re-ordering equations to speed up convergence~\cite{HS-78}, or considering block versions of the Kaczmarz method with relaxation matrices $\Omega_i$ (\cite{EHL-81}, see also \cite{cimmino1938calcolo}).  Block versions of the method allow for over-relaxation parameters of greater than 2 as demonstrated in  \cite{censor2001component,necoara2019faster}, in similar fashion to our results in Section \ref{sec:leaf}.

Relatively recently, choosing the next equation randomly has been shown to dramatically improve the rate of convergence of the algorithm \cite{SV-09a,zouzias2013randomized,needell2014paved,needell2015randomized,chen2018kaczmarz}.  Moreover, this randomized version of the Kaczmarz algorithm has been shown to be comparable to the gradient descent method \cite{NSW16a}.  In our situation, the equations are \emph{a priori} distributed across a network with a fixed topology; this determines the next equation to use to update the estimate and does not allow a choice.  Instead, we demonstrate that the convergence rate can be improved by relaxation parameters greater than 2 in Section \ref{sec:ex}.  

A distributed version of the Kaczmarz algorithm was introduced in \cite{kamath2015distributed}.  The main ideas presented there are very similar to ours:  updated estimates are obtained from prior estimates using the Kaczmarz update with the equations that are available at the node, and distributed estimates are averaged together at a single node (which the authors refer to as a fusion center, for us it is the root of the tree).    Another distributed version was proposed in \cite{liu2014asynchronous}, which has a shared memory architecture.

\section{Consistent Systems}
We prove Theorem \ref{Theorem3.1}, the main result of this section, using a sequence of lemmas. We follow the argument presented in \cite{Hegde2019}, adapting those lemmas for our assumptions on the relaxation parameters. We also direct the reader to the original source \cite{Natterer2001}.


\begin{lemma} \label{genEig}
Let $\mathcal{H}$ be a Hilbert space and $\mathcal{K}$ be a closed subspace of $\mathcal{H}$. Let $U$ be a linear operator on $\mathcal{H}$ with the following properties:
\begin{enumerate}
    \item $U\vec{x}=\vec{x}$ for every $\vec{x}\in\mathcal{K}$,
    \item $\mathcal{K}^\perp$ is an invariant subspace for $U$ (i.e., $U(\mathcal{K}^\perp)\subseteq \mathcal{K}^\perp$),
    \item $\| U|_{\mathcal{K}^\perp}\| < 1$.
\end{enumerate}
Given a sequence $\{\vec{x}_k\}$ in $\mathcal{H}$ such that
\[
\lVert \vec{x}_k\rVert \le 1 \text{ and } \lim_{k\to \infty}\lVert U\vec{x}_k\rVert=1,
\]
it follows that
\[
\lim_{k\to\infty} (I-U)\vec{x}_k=\vec{0}.
\]
\end{lemma}

\begin{proof}
For convenience, we denote $\alpha = \|U|_{\mathcal{K}^\perp}\|$, and let $P$ be the orthogonal projection onto $\mathcal{K}^\perp$. We claim that $\|P\vec{x}_k\| \rightarrow 0$. Indeed, we have
\begin{align*}
1 &= \lim_{k\to\infty}\|U\vec{x}_k\|^2 \\
&= \lim_{k\to\infty}\|U(I - P)\vec{x}_k + UP\vec{x}_k\|^2 \\
&= \lim_{k\to\infty}\left(\|(I - P)\vec{x}_k\|^2 + \|UP\vec{x}_k\|^2\right) \\
&\leq \lim\inf\left(\|(I - P)\vec{x}_k\|^2 + \alpha^2\|P\vec{x}_k\|^2\right) \\
&= \lim\inf\left(\|\vec{x}_k\|^2 - (1 - \alpha^2)\|P\vec{x}_k\|^2\right) \\
&\leq \lim\inf\left(1 - (1 - \alpha^2)\|P\vec{x}_k\|^2\right) \\
&= 1 - (1 - \alpha^2)\lim\sup\|P\vec{x}_k\|^2 \leq 1.
\end{align*}
We therefore observe that $1 - (1 - \alpha^2)\lim\sup\|P\vec{x}_k\|^2 = 1$ so that $\lim\sup\|P\vec{x}_k\| = 0$, as desired. Hence $$\lim_{k\to\infty}(I-U)\vec{x}_k=\lim_{k\to\infty}(I-U)(P\vec{x}_k)=\vec{0}.$$

\end{proof}



\begin{lemma} \label{PGi} Fix an integer $i \in \{1,2,...,c\}$, an enumeration of the subnetworks. Suppose that $\{\vec{x}_k\}$ is a sequence in $\mathbb{C}^d$ such that
\[
\lVert\vec{x}_k \rVert \le 1 \text{ and } \lim _{k \to \infty} \lVert P_{G_i}^{\vec{\omega}} \vec{x}_k \rVert = 1.
\]
It follows that
\[ \lim _ {k \to \infty} (I-P_{G_i}^{\vec{\omega}})\vec{x}_k = \vec{0}.
\]
\end{lemma}

\begin{proof}
Let $\mathcal{K} = \left\{\vec{a}_u\,:\,u\in G_i\right\}^\perp$. The proof consists of simply verifying that $P_{G_i}^{\vec{\omega}}$ satisfies the conditions of Lemma \ref{genEig}.

As observed in Lemma \ref{contractions}, we have that $P^{\vec{\omega}}_{G_i}\vec{x} = \vec{x}$ for every $\vec{x} \in \mathcal{K}$ and that $\mathcal{K}^\perp$ is an invariant subspace for $P^{\vec{\omega}}_{G_i}$. Condition (3) of Lemma \ref{genEig} follows from the assumptions on the relaxation parameters, specifically $\|P^{\vec{\omega}}_{G_i}\vec{x}\| \leq \alpha_i\|\vec{x}\|$ for every $\vec{x} \in \mathcal{K}^\perp$.

\end{proof}

\begin{lemma}\label{PGir}
Fix an integer $i \in \{1,2,...,c\}$, an enumeration of the subnetworks. Suppose that $\{\vec{x}_k\}$ is a sequence in $\mathbb{C}^d$ such that
\[
\|\vec{x}_k\| \leq 1 \text{ and } \lim_{k \rightarrow\infty}\|P^{\vec{\omega}}_{G_i,r}\vec{x}_k\| = 1.
\]
It follows that
\[
\lim_{k\rightarrow\infty}(I - P^{\vec{\omega}}_{G_i,r})\vec{x}_k = \vec{0}.
\]
\end{lemma}

\begin{proof}
Note that
\begin{align*}
(I - P^{\vec{\omega}}_{G_i}P^{\vec{\omega}}_{g_i}\cdots P^{\vec{\omega}}_r)\vec{x}_k = (I - P^{\vec{\omega}}_{g_i}\cdots P^{\vec{\omega}}_r)\vec{x}_k + (I - P^{\vec{\omega}}_{G_i})P^{\vec{\omega}}_{g_i}\cdots P^{\vec{\omega}}_r\vec{x}_k.
\end{align*}
Since $\|P^{\vec{\omega}}_{g_i}\cdots P^{\vec{\omega}}_r\vec{x}_k\| \leq 1$, we have $(I - P^{\vec{\omega}}_{G_i})P^{\vec{\omega}}_{g_i}\cdots P^{\vec{\omega}}_r\vec{x}_k \rightarrow \vec{0}$ from Lemma \ref{PGi}. Hence it suffices to show $(I - P^{\vec{\omega}}_{g_i}\cdots P^{\vec{\omega}}_r)\vec{x}_k \rightarrow \vec{0}$. Consider the path from $r$ to $g_i$, say $r = u_1 \rightarrow u_2 \rightarrow ... \rightarrow u_n = g_i$, and let $\mathcal{K} = \{\vec{a}_{u_j}\,:\,1\leq j\leq n\}^\perp$. We check Lemma \ref{genEig}. Conditions (1) and (2) are straightforward to check, so we only show condition (3). Assume by way of contradiction that $\|P^{\vec{\omega}}_{g_i}\cdots P^{\vec{\omega}}_r|_{\mathcal{K}^\perp}\| = 1$. By continuity and compactness, there then exists a unit vector $\vec{x} \in \mathcal{K}^\perp$ such that $\|P^{\vec{\omega}}_{g_i}\cdots P^{\vec{\omega}}_r\vec{x}\| = 1$. From this observation and Lemma \ref{contraction}, it follows that $\vec{x} \in \mathcal{K}$ so that $\vec{x} = \vec{0}$, which is a contradiction.

\end{proof}

\begin{lemma} \label{Lemma3.6}
Suppose that $\{\vec{x}_k\}$ is a sequence in $\mathbb{C}^d$ such that 
\[
\lVert \vec{x}_k \rVert \le 1 \text{ and } \lim_{k\to \infty} \lVert P^{\vec{\omega}}\vec{x}_k \rVert = 1.
\]
It follows that
\[
\lim_{k \to \infty} (I - P^{\vec{\omega}})\vec{x}_k = \vec{0}.
\]
\end{lemma}

\begin{proof}
Recalling Equation \ref{Pomega}, we note that
\[
(I-P^{\vec{\omega}})\vec{x}_k = \sum_{i=1}^c w(r,g_i)(I-P^{\vec{\omega}}_{G_i,r})\vec{x}_k.
\]
Therefore it suffices to show that the hypotheses of Lemma \ref{PGir} are satisfied. From Lemma \ref{contractions}, we have $\lVert P_{G_i,r}^{\vec{\omega}}\vec{x}_k \rVert \le 1$ and, thus,
\[
1 = \lim_{k\to\infty}\lVert P^{\vec{\omega}}\vec{x}_k \rVert \le \lim\inf\sum _{i=1}^c w(r,g_i) \lVert P_{G_i,r}^{\vec{\omega}}\vec{x}_k \rVert \le 1.
\]
It follows that, for each $i \in \{1,2,...,c\}$,
\[
\lim _{k \to \infty} \lVert P_{G_i,r}^{\vec{\omega}}\vec{x}_k \rVert = 1.
\]

\end{proof}

\begin{proposition}\label{inconsistentlemma1}
If $\|P^{\vec{\omega}}\vec{x}\| = \|\vec{x}\|$, then $\vec{x} \in \mathcal{R}(A^*)^\perp$.
\end{proposition}

\begin{proof}
Note that
\[
\|\vec{x}\| = \left\|\sum_i w(r,g_i)P^{\vec{\omega}}_{G_i}P^{\vec{\omega}}_{g_i}...P^{\vec{\omega}}_r\vec{x}\right\| \leq \sum_i w(r,g_i)\|P^{\vec{\omega}}_{G_i}P^{\vec{\omega}}_{g_i}...P^{\vec{\omega}}_r\vec{x}\| \leq \|\vec{x}\|.
\]
Therefore it follows that $\|P^{\vec{\omega}}_{G_i}P^{\vec{\omega}}_{g_i}...P^{\vec{\omega}}_r\vec{x}\| = \|\vec{x}\|$ for all $i$. Hence $\|P^{\vec{\omega}}_r\vec{x}\| = \|\vec{x}\|$ which, by Lemma \ref{contraction}, implies that $\vec{x} \in \mathcal{N}(S_r)$ and $P^{\vec{\omega}}_r\vec{x} = \vec{x}$. We then inductively find $\vec{x} \in \mathcal{N}(S_{g_i})\cap ...\cap\mathcal{N}(S_r)$,  $P^{\vec{\omega}}_{g_i}\vec{x} = ... = P^{\vec{\omega}}_r\vec{x} = \vec{x}$, and $\|P^{\vec{\omega}}_{G_i}\vec{x}\| = \|\vec{x}\|$. Now let $P$ be the orthogonal projection onto $\{\vec{a}_u\,:\,u\in G_i\}^\perp$. Then, as argued in Lemma \ref{contractions}, we find
\begin{align*}
\|\vec{x}\|^2 &= \|P^{\vec{\omega}}_{G_i}\vec{x}\|^2 \\
&= \|P^{\vec{\omega}}_{G_i}P\vec{x} + P^{\vec{\omega}}_{G_i}(I - P)\vec{x}\|^2 \\
&= \|P\vec{x} + P^{\vec{\omega}}_{G_i}(I - P)\vec{x}\|^2 \\
&= \|P\vec{x}\|^2 + \|P^{\vec{\omega}}_{G_i}(I - P)\vec{x}\|^2 \\
&\leq \|P\vec{x}\|^2 + \alpha_i^2\|(I - P)\vec{x}\|^2 \\
&\leq \|\vec{x}\|^2.
\end{align*}
Therefore $P\vec{x} = \vec{x}$ so that $\vec{x} \in \mathcal{N}(S_u)$ for every $u \in G_i$, which concludes the proof.

\end{proof}

The next lemma is an immediate consequence of Proposition \ref{inconsistentlemma1}.

\begin{lemma} \label{Lemma3.7}
Let $\mathcal{V}$ be the collection of all of the vertices in the network. Then
\[
\mathcal{N}(I-P^{\vec{\omega}}) = \bigcap_{v \in \mathcal{V}} \mathcal{N}(I-P_v).
\]
\end{lemma}




\begin{lemma}\label{Lemma3.8}
Let $\mathcal{V}$ be the collection of all of the vertices in the network. As $k \to \infty$, $(P^{\vec{\omega}})^k$ converges strongly to the orthogonal projection onto
\[
\bigcap_{v \in \mathcal{V}} \mathcal{N}(I-P_v) = \mathcal{N} (A).
\]
\end{lemma}

\begin{proof} 
Using Lemmas \ref{Lemma3.6} and \ref{Lemma3.7} with the observation that $\mathcal{N}(S_v) = \mathcal{N}(I - P_v)$, the proof is identical to the proof of Lemma 3.5 in \cite{Natterer2001}.

\end{proof}

\begin{theorem} \label{Theorem3.1}
If the system of equations $A\vec{x} = \vec{b}$ is consistent, then the sequence of estimates $\{\vec{x}^{(n)}\}$ from the distributed Kaczmarz algorithm given by the recursion
\[
\vec{x}^{(n+1)} = Q^{\vec{\omega}}\vec{x}^{(n)} = \sum _ {\ell \in \mathcal{L}} w(r,\ell) Q^{\vec{\omega}}_\ell\cdots Q^{\vec{\omega}}_r \vec{x}^{(n)},
\]
with admissible relaxation parameters, converges to the solution of minimal norm provided that the initial estimate $\vec{x}^{(0)} \in \mathcal{R}(A^*)$.
\end{theorem}

\begin{proof}
Let $\vec{x}$ be a solution to the system of equations, and let $v$ be any vertex in the network. Then, from Equation \ref{QandP}, we have
\[
\vec{x} = Q_v^{\vec{\omega}}\vec{x} = P^{\vec{\omega}}_v\vec{x} + \omega_v\vec{h}_v.
\]
Let $\vec{y}$ be an arbitrary vector. From Equation \ref{QandP}, again, we find
\[
Q_v^{\vec{\omega}}\vec{y} = P_v^{\vec{\omega}}\vec{y} + \omega_v\vec{h}_v = P_v^{\vec{\omega}}(\vec{y} - \vec{x}) + \vec{x}.
\]
It then immediately follows from this last identity that
\begin{align*}
Q^{\vec{\omega}}\vec{y} &= \sum_{\ell \in \mathcal{L}}w(r,\ell)Q^{\vec{\omega}}_{\ell}\cdots Q^{\vec{\omega}}_r\vec{y} \\
&= \left(\sum_{\ell \in \mathcal{L}}w(r,\ell)P^{\vec{\omega}}_\ell\cdots P^{\vec{\omega}}_r(\vec{y} - \vec{x})\right) + \vec{x} \\
&= P^{\vec{\omega}}(\vec{y} - \vec{x}) + \vec{x}.
\end{align*}
Further, for every positive integer $k$,
\[
(Q^{\vec{\omega}})^k\vec{y} = (P^{\vec{\omega}})^k(\vec{y} - \vec{x}) + \vec{x}.
\]
Then, from Lemma \ref{Lemma3.8}, we have that
\[
(Q^{\vec{\omega}})^k \vec{y} \to T(\vec{y}-\vec{x}) + \vec{x}
\]
where $T$ is the orthogonal projection onto $\mathcal{N}(A)$. Now, if $\vec{y} = \vec{x}^{(0)} \in \mathcal{R}(A^*)$, then $T(\vec{y} - \vec{x}) + \vec{x} = (I - T)\vec{x}$ is the solution of minimal norm, which concludes the proof.

\end{proof}
\section{Inconsistent Systems}\label{inconsistent}
In this section, we show that the distributed Kaczmarz algorithm with admissible relaxation parameters converges regardless of the consistency of the system and that the limit point is an approximation of a weighted least-squares solution when the system of equations is inconsistent. We first develop the relevant theory by following Successive Over-Relaxation (SOR) analysis of the Kaczmarz algorithm as developed in \cite{Natterer2001}.

Let $\ell \in \mathcal{L}$, and suppose $r = u_1 \to u_2 \to \cdots \to u_{p-1} \to u_p = \ell$, the path from $r$ to $\ell$. We denote the initial estimate at $r$ by $\vec{x}_{u_0}$. Then, from the Kaczmarz update, we recursively attain $\vec{x}_{u_j}$, the relaxed projection of $\vec{x}_{u_{j-1}}$ onto the hyperplane given by $\vec{a}_{u_j}^*\vec{x} = b_{u_j}$,
\begin{align}\label{relaxedprojection}
\vec{x}_{u_j} = Q_{u_j}^{\vec{\omega}}\vec{x}_{u_{j-1}} = \vec{x}_{u_{j-1}} + \omega_{u_j}\frac{b_{u_j} - \vec{a}_{u_j}^*\vec{x}_{u_{j-1}}}{\|\vec{a}_{u_j}\|^2}\vec{a}_{u_j}.
\end{align}
Hence, there exist complex scalars $\{c_k\}_{k=1}^p$ such that, for all $j$,
\begin{align}\label{sumrepresentation}
\vec{x}_{u_j} = \vec{x}_{u_0} + \sum_{k=1}^jc_k\vec{a}_{u_k}.
\end{align}
Substituting Equation \ref{sumrepresentation} into Equation \ref{relaxedprojection},
\begin{align}\label{coefficients}
c_j = \omega_{u_j}\dfrac{b_{u_j} - \vec{a}_{u_j}^*\vec{x}_{u_0} - \sum\limits_{k = 1}^{j-1}c_k\vec{a}_{u_j}^*\vec{a}_{u_k}}{\|\vec{a}_{u_j}\|^2}.
\end{align}
We can then consolidate Equation \ref{coefficients} for all $j$ into the matrix equation
\begin{align}\label{coefficientsvector}
D_\ell\vec{c} = \Omega_\ell(\vec{b}_\ell - A_\ell\vec{x}_{u_0} - L_\ell\vec{c})
\end{align}
where $\vec{c} = (c_1,c_2,...,c_p)^T$ and $D_\ell$, $\Omega_\ell$, $\vec{b}_\ell$, $L_\ell$ and $A_\ell$ are as follows:
\[
D_\ell = \begin{pmatrix}
\|\vec{a}_{u_1}\|^2 & 0 & \hdots & 0 \\
0 & \|\vec{a}_{u_2}\|^2 & \hdots & 0 \\
\vdots & \vdots & \ddots & \vdots \\
0 & 0 & \hdots & \|\vec{a}_{u_p}\|^2
\end{pmatrix},\,\Omega_\ell = \begin{pmatrix}
\omega_{u_1} & 0 & \hdots & 0 \\
0 & \omega_{u_2} & \hdots & 0 \\
\vdots & \vdots & \ddots & \vdots \\
0 & 0 & \hdots & \omega_{u_p}
\end{pmatrix},
\]
\[
\vec{b}_{\ell} = \begin{pmatrix}
b_{u_1} \\
b_{u_2} \\
\vdots \\
b_{u_p}
\end{pmatrix},\,L_\ell = \begin{pmatrix}
0 & 0 & 0 & \hdots & 0 & 0 \\
\vec{a}_{u_2}^*\vec{a}_{u_1} & 0 & 0 & \hdots & 0 & 0 \\
\vec{a}_{u_3}^*\vec{a}_{u_1} & \vec{a}_{u_3}^*\vec{a}_{u_2} & 0 & \hdots & 0 & 0 \\
\vec{a}_{u_4}^*\vec{a}_{u_1} & \vec{a}_{u_4}^*\vec{a}_{u_2} & \vec{a}_{u_4}^*\vec{a}_{u_3} & \hdots & 0 & 0 \\
\vdots & \vdots & \vdots & \ddots & \vdots & \vdots \\
\vec{a}_{u_p}^*\vec{a}_{u_1} & \vec{a}_{u_p}^*\vec{a}_{u_2} & \vec{a}_{u_p}^*\vec{a}_{u_3} & \hdots & \vec{a}_{u_p}^*\vec{a}_{u_{p-1}} & 0
\end{pmatrix},\,A_\ell = \begin{pmatrix}
\vec{a}_{u_1}^* \\
\vec{a}_{u_2}^* \\
\vdots \\
\vec{a}_{u_p}^*
\end{pmatrix}.
\]
Altogether, from Equations \ref{sumrepresentation} and \ref{coefficientsvector}, respectively, we may express the iterate $\vec{x}_\ell^{(n)}$ at the leaf $\ell$ given the initial vector $\vec{x}^{(n)}$ at the root in terms of the scalar vector $\vec{c}$,
\begin{align*}
\vec{x}_\ell^{(n)} &= \vec{x}^{(n)} + A_\ell^*\vec{c}, \\
\vec{c} &= (D_\ell + \Omega_\ell L_\ell)^{-1}\Omega_\ell\left(\vec{b}_\ell - A_\ell\vec{x}^{(n)}\right).
\end{align*}
We eliminate the scalar vector and attain
\begin{align*}
\vec{x}_\ell^{(n)} = (I - A_\ell^*(D_\ell + \Omega_\ell L_\ell)^{-1}\Omega_\ell A_\ell)\vec{x}^{(n)} + A_\ell^*(D_\ell + \Omega_\ell L_\ell)^{-1}\Omega_\ell\vec{b}_\ell.
\end{align*}
We then aggregate the leaf operators as follows:
\[
D = \begin{pmatrix}
D_{\ell_1} & 0 & ... & 0 \\
0 & D_{\ell_2} & ... & 0 \\
\vdots & \vdots & \ddots & \vdots \\
0 & 0 & ... & D_{\ell_t}
\end{pmatrix}, \Omega = \begin{pmatrix}
\Omega_{\ell_1} & 0 & ... & 0 \\
0 & \Omega_{\ell_2} & ... & 0 \\
\vdots & \vdots & \ddots & \vdots \\
0 & 0 & ... & \Omega_{\ell_t}
\end{pmatrix},
\]
\[
\vec{\mathfrak{b}} = \begin{pmatrix}
\vec{b}_{\ell_1} \\
\vec{b}_{\ell_2} \\
\vdots \\
\vec{b}_{\ell_t}
\end{pmatrix}, L = \begin{pmatrix}
L_{\ell_1} & 0 & ... & 0 \\
0 & L_{\ell_2} & ... & 0 \\
\vdots & \vdots & \ddots & \vdots \\
0 & 0 & ... & L_{\ell_t}
\end{pmatrix}, \mathcal{A} = \begin{pmatrix}
A_{\ell_1} \\
A_{\ell_2} \\
\vdots \\
A_{\ell_t}
\end{pmatrix},
\]
\[
W = \begin{pmatrix}
w(r,\ell_1)I_{\text{dim}(\Omega_{\ell_1})} & 0 & ... & 0 \\
0 & w(r,\ell_2)I_{\text{dim}(\Omega_{\ell_2})} & ... & 0 \\
\vdots & \vdots & \ddots & \vdots \\
0 & 0 & ... & w(r,\ell_t)I_{\text{dim}(\Omega_{\ell_t})}
\end{pmatrix}.
\]
The estimate obtained from the pooling stage of the $n$th iteration can be expressed in terms of these matrices,
\begin{align}\label{groupiterate}
\vec{x}^{(n+1)} = \sum_{\ell \in \mathcal{L}}w(r,\ell)\vec{x}_\ell^{(n)} = B^{\vec{\omega}}\vec{x}^{(n)} + \vec{b}^{\vec{\omega}}
\end{align}
where
\begin{align}
B^{\vec{\omega}} &= I - \mathcal{A}^*(D + \Omega L)^{-1}W\Omega \mathcal{A}, \\
\vec{b}^{\vec{\omega}} &= \mathcal{A}^*(D + \Omega L)^{-1}W\Omega\vec{\mathfrak{b}}.
\end{align}

Note that there exists a vector $\vec{h}$ such that $Q^{\vec{\omega}}\vec{x} = P^{\vec{\omega}}\vec{x} + \vec{h}$ for every $\vec{x}$. Then, from Equation \ref{groupiterate} and the linearity of $P^{\vec{\omega}}$ and $B^{\vec{\omega}}$, we have $B^{\vec{\omega}} = P^{\vec{\omega}}$ and $\vec{b}^{\vec{\omega}} = \vec{h}$.

\begin{proposition}\label{eigenvalues}
Suppose $B^{\vec{\omega}}\vec{x} = \lambda\vec{x}$ for some $\vec{x} \neq \vec{0}$. Then $\lambda = 1$ or $|\lambda| < 1$, and
\begin{enumerate}
    \item $\lambda = 1$ if and only if $\vec{x} \in \mathcal{R}(A^*)^\perp$,
    \item $|\lambda| < 1$ if and only if $\vec{x} \in \mathcal{R}(A^*)$.
\end{enumerate}
\end{proposition}

\begin{proof}
Suppose $P^{\vec{\omega}}\vec{x} = \lambda\vec{x}$ for some $\vec{x} \neq \vec{0}$. By Lemma \ref{contractions}, we note that $|\lambda| \leq 1$. Let $P$ be the orthogonal projection onto $\mathcal{R}(A^*)^\perp$. Then we find
\[
\lambda P\vec{x} + \lambda (I - P)\vec{x} = \lambda\vec{x} = P^{\vec{\omega}}\vec{x} = P^{\vec{\omega}}P\vec{x} + P^{\vec{\omega}}(I - P)\vec{x} = P\vec{x} + P^{\vec{\omega}}(I - P)\vec{x}.
\]
By uniqueness of the decomposition in $\mathcal{R}(A^*)\oplus\mathcal{R}(A^*)^\perp$, we have
\begin{align*}
    P\vec{x} &= \lambda P\vec{x}, \\
    P^{\vec{\omega}}(I - P)\vec{x} &= \lambda(I - P)\vec{x}.
\end{align*}
If $\lambda \neq 1$, then $P\vec{x} = \vec{0}$ so that $\vec{x} = (I - P)\vec{x} \in \mathcal{R}(A^*)$. From this observation and Proposition \ref{inconsistentlemma1}, we find that $|\lambda| < 1$. Now suppose $\lambda = 1$. Then, by Proposition \ref{inconsistentlemma1}, $(I - P)\vec{x} \in \mathcal{R}(A^*)^\perp$ so that $\vec{x} = P\vec{x} \in \mathcal{R}(A^*)^\perp$. The sufficient statement of (1) is straightforward, and (2) follows.

\end{proof}

\begin{lemma}\label{Convergence}
Let $\vec{x}^{(0)} \in \mathcal{R}(A^*)$. The sequence $\{\vec{x}^{(n)}\}$ converges to the fixed point of the mapping $\vec{x} \in \mathcal{R}(A^*) \mapsto B^{\vec{\omega}}\vec{x} + \vec{b}^{\vec{\omega}}$. Precisely, the sequence converges to
\[
(I - B^{\vec{\omega}})|_{\mathcal{R}(A^*)}^{-1}\vec{b}^{\vec{\omega}} = \sum_{j=0}^\infty(B^{\vec{\omega}})^j\vec{b}^{\vec{\omega}}.
\]

\end{lemma}

\begin{proof}
Throughout the proof, we assume that every operator is restricted to $\mathcal{R}(A^*)$. From Proposition \ref{eigenvalues}, there exists an induced matrix norm $\|\cdot\|$ such that $\|B^{\vec{\omega}}\| < 1$. Note that, with respect to this norm, $(B^{\vec{\omega}})^n$ converges to the zero matrix and $(B^{\vec{\omega}})^{n-1} + ... + B^{\vec{\omega}} + I$ converges to the matrix $(I - B^{\vec{\omega}})^{-1}$. Then
\[
\vec{x}^{(n)} = (B^{\vec{\omega}})^n\vec{x}^{(0)} + ((B^{\vec{\omega}})^{n-1} + ... + B^{\vec{\omega}} + I)\vec{b}^{\vec{\omega}} \to (I - B^{\vec{\omega}})^{-1}\vec{b}^{\vec{\omega}} =: \vec{z}.
\]
Note that $\vec{z} \in \mathcal{R}(A^*)$ and that $\vec{z} = B^{\vec{\omega}}\vec{z} + \vec{b}^{\vec{\omega}}$, as desired.

\end{proof}

\begin{remark}
We observe that, in general, the sequence $\{\vec{x}^{(n)}\}$ converges to
\begin{align}\label{limit}
\vec{y} = \sum_{j=0}^\infty(B^{\vec{\omega}})^j\vec{b}^{\vec{\omega}} + P\vec{x}^{(0)}
\end{align}
where $P$ is the orthogonal projection onto $\mathcal{N}(A)$. Hence, it is novel to choose $\vec{x}^{(0)} \in \mathcal{R}(A^*)$ (e.g., $\vec{x} = \vec{0}$) so that the norm of the vector in Equation \ref{limit} is minimized.

\end{remark}

\begin{theorem}\label{mainthm}
Let $\vec{x}^{(0)} \in \mathcal{R}(A^*)$. The distributed Kaczmarz algorithm with admissible relaxation parameters converges to the vector $\vec{y}$ in Equation \ref{limit}. If the system is inconsistent and $\Omega = s\Omega_1$ where $s \in (0,1]$, then $\vec{y} = \vec{y}_M + o(s)$ where $\vec{y}_M$ minimizes the functional
\[
\vec{x} \in \mathcal{R}(A^*) \mapsto \langle D^{-1}W\Omega_1(\vec{\mathfrak{b}} - \mathcal{A}\vec{x}),\vec{\mathfrak{b}} - \mathcal{A}\vec{x}\rangle.
\]
\end{theorem}

\begin{proof}
With Lemma \ref{Convergence}, the proof is similar to the proof of Theorem V.3.9. in \cite{Natterer2001}. Nonetheless, we provide a self-contained proof for clarification of our adaptation. First, by Lemma \ref{Convergence}, we have that the sequence $\{\vec{x}^{(n)}\}$ converges to the vector $\vec{y}$ satisfying $\vec{y} = B^{\vec{\omega}}\vec{y} + \vec{b}^{\vec{\omega}}$, that is
\begin{align}\label{unique}
\mathcal{A}^*(D + \Omega L)^{-1}W\Omega\mathcal{A}\vec{y} = \mathcal{A}^*(D + \Omega L)^{-1}W\Omega\vec{\mathfrak{b}}.
\end{align}
Note that $\vec{y}_M$ minimizes $\|D^{-1/2}W^{1/2}\Omega_1^{1/2}(\vec{\mathfrak{b}} - \mathcal{A}\vec{x})\|$ if and only if
\[
(D^{-1/2}W^{1/2}\Omega_1^{1/2}\mathcal{A})^*(D^{-1/2}W^{1/2}\Omega_1^{1/2}\mathcal{A})\vec{y}_M = (D^{-1/2}W^{1/2}\Omega_1^{1/2}\mathcal{A})^*D^{-1/2}W^{1/2}\Omega_1^{1/2}\vec{\mathfrak{b}}
\]
(see Theorem 1.1 of IV.1 in \cite{Natterer2001}), that is
\begin{align}\label{eq1}
\mathcal{A}^*D^{-1}W\Omega_1\mathcal{A}\vec{y}_M = \mathcal{A}^*D^{-1}W\Omega_1\vec{\mathfrak{b}}.
\end{align}
Substituting $\Omega = s\Omega_1$ into Equation \ref{unique}, we have
\begin{align}\label{eq2}
\mathcal{A}^*(D + s\Omega_1L)^{-1}W\Omega_1\mathcal{A}\vec{y} = \mathcal{A}^*(D + s\Omega_1L)^{-1}W\Omega_1\vec{\mathfrak{b}}.
\end{align}
From Equations \ref{eq1} and \ref{eq2}, we observe that $\vec{y} = \vec{y}_M + o(s)$, as desired.

\end{proof}

\begin{remark}
The minimizer of the functional in Theorem \ref{mainthm} is the weighted least-squares solution of
\[
\vec{x}\in\mathcal{R}(A^*)\mapsto\sum_{v \in \mathcal{V}}(\Omega_1)_v\left(\sum_{\ell\,:\,v\preceq\ell}w(r,\ell)\right)\dfrac{|b_v - \vec{a}^*_v\vec{x}|^2}{\|\vec{a}_v\|^2}.
\]
We note that there is a trade-off between the convergence rate of the algorithm and the approximation error; that is, the algorithm converges more slowly as $s$ approaches zero.

\end{remark}
\section{Leaf Subnetworks} \label{sec:leaf}
In this section, we consider the particular situation in which the subnetworks consist of leaves. We derive a concise expression for the norm of $P^{\vec{\omega}}_{G_i}$ restricted to the subspace $\mathcal{H}_i := \text{span}\{\vec{a}_u\,:\,u \in G_i\}$ and provide sufficient upper-bounds on the relaxation parameters for the vertices in $G_i$ to guarantee admissibility. We recall that the Gram matrix $\mathcal{G}(\vec{x}_1,\vec{x}_2,...,\vec{x}_t)$ is the $t \times t$ matrix of inner-products,
\[
\mathcal{G}(\vec{x}_1,\vec{x}_2,...,\vec{x}_t) = \begin{pmatrix}
\langle\vec{x}_1,\vec{x}_1\rangle & \langle\vec{x}_1,\vec{x}_2\rangle & ... & \langle\vec{x}_1,\vec{x}_t\rangle \\
\langle\vec{x}_2,\vec{x}_1\rangle & \langle\vec{x}_2,\vec{x}_2\rangle & ... & \langle\vec{x}_2,\vec{x}_t\rangle \\
\vdots & \vdots & \ddots & \vdots \\
\langle\vec{x}_t,\vec{x}_1\rangle & \langle\vec{x}_t,\vec{x}_2\rangle & ... & \langle\vec{x}_t,\vec{x}_t\rangle \\
\end{pmatrix}.
\]
We further denote the diagonal matrix $\mathcal{D}_i$ associated with the leaf subnetwork $G_i = \{\ell_{i,1},\ell_{i,2},...,\ell_{i,t_i}\}$ by
\[
\mathcal{D}_i = \begin{pmatrix}
\dfrac{w(g_i,\ell_{i,1})\omega_{\ell_{i,1}}}{\|\vec{a}_{\ell_{i,1}}\|^2} & 0 & \hdots & 0 \\
0 & \dfrac{w(g_i,\ell_{i,2})\omega_{\ell_{i,2}}}{\|\vec{a}_{\ell_{i,2}}\|^2} & \hdots & 0 \\
\vdots & \vdots & \ddots & \vdots \\
0 & 0 & \hdots & \dfrac{w(g_i,\ell_{i,t_i})\omega_{\ell_{i,t_i}}}{\|\vec{a}_{\ell_{i,t_i}}\|^2}
\end{pmatrix}.
\]

We denote the spectrum (collection of eigenvalues) of a matrix $A$ by $\sigma(A)$, and we denote its spectral radius by $\rho(A) = \max\{|\lambda|\,:\,\lambda \in \sigma(A)\}$.

\begin{theorem}\label{PGnorm}
Suppose $G_i = \{\ell_{i,1},\ell_{i,2},...,\ell_{i,t_i}\}$. Then
\[
\|P_{G_i}^{\vec{\omega}}|_{\mathcal{H}_i}\| = \max\{|1 - \lambda|\,:\,\lambda \in \sigma(\mathcal{D}_i\mathcal{G}(\vec{a}_{\ell_{i,1}},\vec{a}_{\ell_{i,2}},...,\vec{a}_{\ell_{i,t_i}}))\smallsetminus\{0\}\}.
\]
\end{theorem}

\begin{proof}
From Equations \ref{operatorP}, \ref{operatorPw} and
\ref{eq15}, we have
\begin{align}\label{PGmatrix}
P^{\vec{\omega}}_{G_i} = I - \sum_{j=1}^{t_i}\dfrac{w(g_i,\ell_{i,j})\omega_{\ell_{i,j}}}{\|\vec{a}_{\ell_{i,j}}\|^2}\vec{a}_{\ell_{i,j}}\vec{a}_{\ell_{i,j}}^*.
\end{align}
Now let $K_{G_i} := \sqrt{\mathcal{D}_i}(\vec{a}_{\ell_{i,1}},\vec{a}_{\ell_{i,2}},...,\vec{a}_{\ell_{i,t_i}})^*$. Then, Equation \ref{PGmatrix} may be expressed as $P^{\vec{\omega}}_{G_i} = I - K_{G_i}^*K_{G_i}$. Note that $\mathcal{H}_i$ is an invariant subspace for $K_{G_i}^*K_{G_i}$. Hence, from the spectral mapping theorem, we find
\[
\sigma(P_{G_i}^{\vec{\omega}}|_{\mathcal{H}_i}) = 1 - \sigma(K_{G_i}^*K_{G_i}|_{\mathcal{H}_i}).
\]
We claim that $\sigma(K_{G_i}^*K_{G_i}|_{\mathcal{H}_i})$ is precisely the collection of all of the nonzero eigenvalues of $K_{G_i}^*K_{G_i}$. Suppose, to the contrary, that there exists a nonzero vector $\vec{x} \in \mathcal{H}_i$ such that $K_{G_i}^*K_{G_i}\vec{x} = 0$. Then $K_{G_i}\vec{x} \in R(K_{G_i})\cap\mathcal{N}(K_{G_i}^*)$ implying $K_{G_i}\vec{x} = 0$, yet this leads to the contradiction that $\vec{x} \in \mathcal{H}_i \cap \mathcal{H}_i^\perp$ or $\vec{x} = 0$. It is well-known that $K_{G_i}^*K_{G_i}$ and $K_{G_i}K_{G_i}^*$ have the same nonzero eigenvalues and
\begin{align*}
\sigma(K_{G_i}K_{G_i}^*) &= \sigma\left(\sqrt{\mathcal{D}_i}\mathcal{G}(\vec{a}_{\ell_{i,1}},\vec{a}_{\ell_{i,2}},...,\vec{a}_{\ell_{i,t_i}})^T\sqrt{\mathcal{D}_i}\right) \\
&= \sigma(\mathcal{D}_i\mathcal{G}(\vec{a}_{\ell_{i,1}},\vec{a}_{\ell_{i,2}},...,\vec{a}_{\ell_{i,t_i}})),
\end{align*}
concluding the proof.

\end{proof}

\begin{corollary}\label{PGnormcor}
Suppose $G_i = \{\ell_{i,1},\ell_{i,2},...,\ell_{i,t_i}\}$. If
\[
0 < \omega_{\ell_{i,j}} < \dfrac{2\|\vec{a}_{\ell_{i,j}}\|^2}{w(g_i,\ell_{i,j})\rho(\mathcal{G}(\vec{a}_{\ell_{i,1}},\vec{a}_{\ell_{i,2}},...,\vec{a}_{\ell_{i,t_i}}))} \qquad\hbox{for all}\qquad 1 \leq j \leq t_i,
\]
then $\|P_{G_i}^{\vec{\omega}}|_{\mathcal{H}_i}\| < 1$.
\end{corollary}

\begin{proof}
Since $\mathcal{D}_i$ and $\mathcal{G}(\vec{a}_{\ell_{i,1}},\vec{a}_{\ell_{i,2}},...,\vec{a}_{\ell_{i,t_i}})$ are positive semi-definite matrices, the eigenvalues of $\mathcal{D}_i\mathcal{G}(\vec{a}_{\ell_{i,1}},\vec{a}_{\ell_{i,2}},...,\vec{a}_{\ell_{i,t_i}})$ are nonnegative. Therefore, by Theorem \ref{PGnorm}, it suffices to show $\lambda < 2$ for $\lambda \in \sigma(\mathcal{D}_i\mathcal{G}(\vec{a}_{\ell_{i,1}},\vec{a}_{\ell_{i,2}},...,\vec{a}_{\ell_{i,t_i}}))$. Let $j$ be the index for the largest diagonal entry in $\mathcal{D}_i$. By Theorem 8.12 in \cite{Zhang}, we have
\begin{align*}
\rho(\mathcal{D}_i\mathcal{G}(\vec{a}_{\ell_{i,1}},\vec{a}_{\ell_{i,2}},...,\vec{a}_{\ell_{i,t_i}})) \leq \dfrac{w(g_i,\ell_{i,j})\omega_{\ell_{i,j}}}{\|\vec{a}_{\ell_{i,j}}\|^2}\rho(\mathcal{G}(\vec{a}_{\ell_{i,1}},\vec{a}_{\ell_{i,2}},...,\vec{a}_{\ell_{i,t_i}})) < 2,
\end{align*}
as desired.

\end{proof}

\begin{remark}\label{spectralnormone}
It is not unusual to require that the rows of $A$ are normalized (i.e., $\|\vec{a}_u\| = 1$ for all $u$). Further, for the case that $\rho(\mathcal{G}(\vec{a}_{\ell_{i,1}},\vec{a}_{\ell_{i,2}},...,\vec{a}_{\ell_{i,t_i}})) \approx 1$, the relaxation parameters for the vertices in $G_i$ are admissible if
\[
\omega_{\ell_{i,j}} \lesssim \dfrac{2}{w(g_i,\ell_{i,j})} \qquad\hbox{for all}\qquad 1 \leq j \leq t_i.
\]
This upper-bound is greater than the usual bound in literature and can be drastically larger than $2$, depending on the associated weights. For example, if the weights are uniformly distributed, then the upper-bound is $2t_i \geq 2$.
\end{remark}

We end this section by observing that it is necessary and sufficient to check that $\Omega_1$ satisfies the admissibility conditions in Theorem \ref{mainthm} when the subnetwork consists of only leaves. We note that this need not hold for other subnetworks.

\begin{theorem}
Suppose $G_i = \{\ell_{i,1},\ell_{i,2},...,\ell_{i,t_i}\}$. Let $\Omega = s\Omega_1$ for some $s \in (0,1]$ as in Theorem \ref{mainthm}. If $\Omega_1$ satisfies the admissibility conditions, then $\Omega$ satisfies the admissibility conditions.
\end{theorem}

\begin{proof}
We check condition (2) in Definition \ref{D:admissible}. Let $\vec{x} \in \mathcal{H}_i$. Then
\begin{align*}
\|P^{\Omega}_{G_i}\vec{x}\| &= \left\|\sum_{j=1}^{t_i}w(g_i,\ell_{i,j})P^{\Omega}_{\ell_{i,j}}\vec{x}\right\| \\
&= \left\|\sum_{j=1}^{t_i}w(g_i,\ell_{i,j})\left[(1 - s)I + sP^{\Omega_1}_{\ell_{i,j}}\right]\vec{x}\right\| \\
&= \left\|(1 - s)\vec{x} + sP^{\Omega_1}_{G_i}\vec{x}\right\| \\
&\leq (1 - s)\|\vec{x}\| + s\alpha_i\|\vec{x}\| \\
&= [(1 - s)1 + s\alpha_i]\|\vec{x}\|,
\end{align*}
where the coefficient is strictly less than one as it is a convex sum of $1$ and $\alpha_i$.
\end{proof}

\section{Experiments} \label{sec:ex}

In this section we implement our algorithm on various kinds of distributed networks corresponding to randomly generated systems of equations and systems perturbed from an orthogonal coefficient matrix. The latter illustrates the point of Remark \ref{spectralnormone}. Specifically, we analyze two scenarios: (1) comparing different subnetwork structures for a given network and (2) comparing different network structures for a given system of equations.

For the first experiment, we consider a 7-node binary network and compare leaf subnetworks to extended subnetworks as depicted in Figure \ref{fig:subnetworks}. We assign the relaxation parameters as follows: set $\omega_v = 1.5$ if the node $v$ is not associated with a subnetwork; set $\omega_v = \omega$ if the node $v$ belongs to a subnetwork. Then we calculate the spectral radius of the operator $P^{\vec{\omega}}$ as a function of $\omega$. For a baseline, we include the spectral radius of the network with no subnetwork structures in this set-up, which we label uniform.

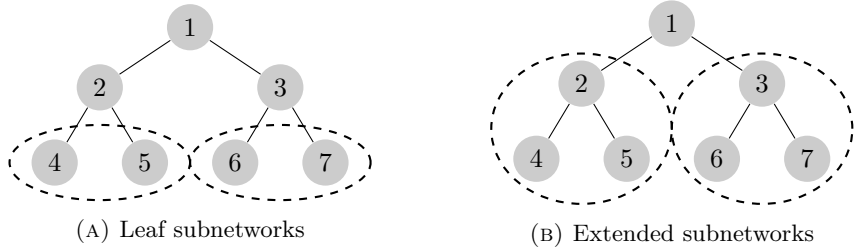
\begin{figure*}[h!]
    \centering
    \begin{subfigure}[h]{0.5\textwidth}
        \centering
        
        \begin{tikzpicture}    

        \def\w{0.8}
        \def\u{0.4}
                
        \node[circle,fill=black!20!white] (1d) at (-3*\w,-6*\u) {1};
        \node[circle,fill=black!20!white] (2d) at (-4.5*\w,-8*\u) {2};
        \node[circle,fill=black!20!white] (3d) at (-1.5*\w,-8*\u) {3};
        \node[circle,fill=black!20!white] (4d) at (-5.25*\w,-10.5*\u) {4};
        \node[circle,fill=black!20!white] (5d) at (-3.75*\w,-10.5*\u) {5};
        \node[circle,fill=black!20!white] (6d) at (-2.25*\w,-10.5*\u) {6};
        \node[circle,fill=black!20!white] (7d) at (-0.75*\w,-10.5*\u) {7};
        
        \draw (1d) to (3d);
        \draw (1d) to (2d);
        \draw (2d) to (4d);
        \draw (2d) to (5d);
        \draw (3d) to (6d);
        \draw (3d) to (7d);
        
        \draw[black,thick,dashed] (-4.5*\w,-10.5*\u) 
        ellipse (1.5*\w cm and 1.25*\u cm);
        \draw[black,thick,dashed] (-1.5*\w,-10.5*\u) 
        ellipse (1.5*\w cm and 1.25*\u cm);
        
        \end{tikzpicture}
        
        \caption{Leaf subnetworks}
    \end{subfigure} %
    ~
    \begin{subfigure}[h]{0.5\textwidth}
        \centering
        
        \begin{tikzpicture}    

        \def\w{0.8}
        \def\u{0.4}
                
        \node[circle,fill=black!20!white] (1d) at (-3*\w,-6*\u) {1};
        \node[circle,fill=black!20!white] (2d) at (-4.5*\w,-8*\u) {2};
        \node[circle,fill=black!20!white] (3d) at (-1.5*\w,-8*\u) {3};
        \node[circle,fill=black!20!white] (4d) at (-5.25*\w,-10.5*\u) {4};
        \node[circle,fill=black!20!white] (5d) at (-3.75*\w,-10.5*\u) {5};
        \node[circle,fill=black!20!white] (6d) at (-2.25*\w,-10.5*\u) {6};
        \node[circle,fill=black!20!white] (7d) at (-0.75*\w,-10.5*\u) {7};
        
        \draw (1d) to (3d);
        \draw (1d) to (2d);
        \draw (2d) to (4d);
        \draw (2d) to (5d);
        \draw (3d) to (6d);
        \draw (3d) to (7d);
        
        \draw[black,thick,dashed] (-4.5*\w,-9.5*\u) 
        ellipse (1.5*\w cm and 2.5*\u cm);
        \draw[black,thick,dashed] (-1.5*\w,-9.5*\u) 
        ellipse (1.5*\w cm and 2.5*\u cm);
        
        \end{tikzpicture}
        
        \caption{Extended subnetworks}
    \end{subfigure}
    \caption{The 7-node binary network with its subnetworks}
    \label{fig:subnetworks}
\end{figure*}

\begin{figure}[h!]
    \centering
    \includegraphics[scale = 0.325]{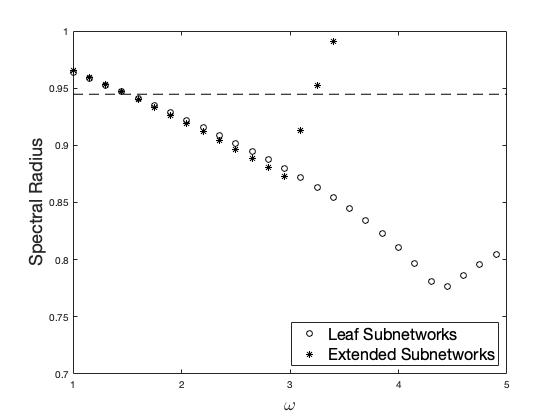}
    \caption{Spectral radius of $P^{\vec{\omega}}$ for two subnetwork structures. The dashed line represents a network with uniformly distributed relaxation parameters $\omega = 1.5$.}
\end{figure}

The numerical experiments suggest that the leaf subnetwork structures are more practical than the extended subnetwork structures for two reasons. In general, the spectral radius of $P^{\vec{\omega}}$ is decreasing for $\omega$ slightly larger than $1.5$ and is, therefore, comparatively smaller than the baseline established by the uniform case in which all of the parameters are set to 1.5. In this situation, we find that the spectral radius tends to be smaller than the baseline for relatively large relaxation parameters in the case of the leaf subnetwork structures and less so in the case of the extended subnetwork structures. This implies that parameter selection is more reliable for leaf subnetworks than for their extended counterparts. Second, the spectral radius is often smaller for leaf subnetworks when the parameters are large. We believe that these observations are a consequence of the pooling stage which is a poor method of producing the next iterate in the distributed Kaczmarz algorithm from the leaf estimates. The depth of the extended network increases the number of overrelaxed projections, often leading to adverse results in the pooling stage.



\begin{figure*}[h!]
    \centering
    \begin{subfigure}[t]{0.5\textwidth}
        \centering
        \begin{tikzpicture}    

        \def\w{0.8}
        \def\u{0.4}
        
        \node[circle,fill=black!20!white] (1d) at (-3*\w,-6*\u) {1};
        \node[circle,fill=black!20!white] (2d) at (-4.5*\w,-8*\u) {2};
        \node[circle,fill=black!20!white] (3d) at (-1.5*\w,-8*\u) {3};
        \node[circle,fill=black!20!white] (4d) at (-5.25*\w,-10.5*\u) {4};
        \node[circle,fill=black!20!white] (5d) at (-3.75*\w,-10.5*\u) {5};
        
        \draw (1d) to (3d);
        \draw (1d) to (2d);
        \draw (2d) to (4d);
        \draw (2d) to (5d);

        \end{tikzpicture}
        \caption{Network I}
    \end{subfigure} %
    ~
    \begin{subfigure}[t]{0.5\textwidth}
        \centering
        \begin{tikzpicture}    

        \def\w{0.8}
        \def\u{0.4}
                
        \node[circle,fill=black!20!white] (1d) at (-3*\w,-6*\u) {1};
        \node[circle,fill=black!20!white] (2d) at (-4.5*\w,-8*\u) {2};
        \node[circle,fill=black!20!white] (3d) at (-1.5*\w,-8*\u) {3};
        \node[circle,fill=black!20!white] (4d) at (-4.5*\w,-10.5*\u) {4};
        \node[circle,fill=black!20!white] (5d) at (-1.5*\w,-10.5*\u) {5};
        
        \draw (1d) to (3d);
        \draw (1d) to (2d);
        \draw (2d) to (4d);
        \draw (3d) to (5d);
        
        \end{tikzpicture}
        \caption{Network II}
    \end{subfigure}
    \caption{Two networks for a system of five equations}
    \label{fig:5nodenetworks}
\end{figure*}
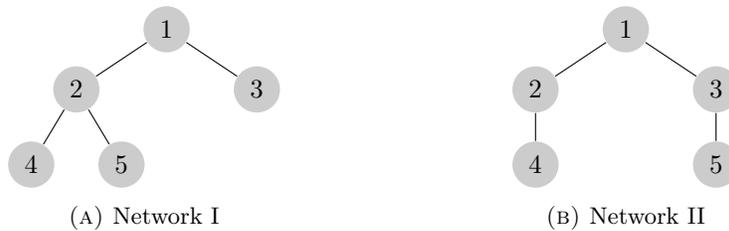


For the second experiment, we consider the different network structures given in Figure \ref{fig:5nodenetworks} for a system of five equations. We compare the network structures for two kinds of systems: (1) entries of $A$ are randomly selected from a uniform distribution over $[0,1]$ and (2) $A$ is nearly orthogonal by perturbing the identity. Further, the entries of $\vec{b}$ are also randomly selected from a uniform distribution over $[0,1]$. We present results of numerical experiments for the nearly orthogonal system in Table \ref{table:orthogonal} and for the random system in Table \ref{table:random}. We include the optimal relaxation parameters that yield the minimum spectral radius along with an error estimate of an iterate using the optimal parameters. Figure \ref{fig:specradi_orth} shows how the spectral radius varies with respect to the relaxation parameter for networks I and II with leaf subnetworks. For network I, $\omega_1$ is on node 3, and $\omega_2$ is on the leaf subnetwork composed of nodes 4 and 5. For network II, $\omega_1$ is on node 5, and $\omega_2$ is on node 4.

\begin{table}[h!]
    \centering
    \begin{tabular}{c|c|c|c|c|c}
        & \multicolumn{3}{c|}{Leaf subnetworks} & \multicolumn{2}{c}{Uniform} \\
        \hline
        Network Type & ($\omega_1$, $\omega_2$)$_{\text{opt}}$ & $\rho(P^{\vec{\omega}})$ & $\|A\vec{x}^{(10)} - \vec{b}\|$ & $\rho(P^{\vec{\omega}})$ & $\|A\vec{x}^{(10)} - \vec{b}\|$ \\
        \hline
        I & (2.27,\,3.93) & 0.36532 & 3.479e-4 & 0.66617 & 3.6441e-3 \\
        \hline
        II & (1.49,\,2.52) & 0.37492 & 3.4554e-4 & 0.47598 & 5.7009e-4 \\
        \hline
    \end{tabular}
    \vspace{0.5cm}
    \caption{Comparing networks I and II for a nearly orthogonal system}
    \label{table:orthogonal}
\end{table}

\begin{table}[h!]
    \centering
    \begin{tabular}{c|c|c|c|c|c}
        & \multicolumn{3}{c|}{Leaf subnetworks} & \multicolumn{2}{c}{Uniform} \\
        \hline
        Network Type & ($\omega_1$, $\omega_2$)$_{\text{opt}}$ & $\rho(P^{\vec{\omega}})$ & $\|A\vec{x}^{(1500)} - \vec{b}\|$ & $\rho(P^{\vec{\omega}})$ & $\|A\vec{x}^{(1500)} - \vec{b}\|$ \\
        \hline
        I & (7.92,\,8.06) & 0.98844 & 1.5743e-8 & 0.99626 & 1.7049e-3 \\
        \hline
        II & (4.57,\,3.90) & 0.99512 & 8.3191e-4 & 0.99619 & 1.3288e-3 \\
        \hline
    \end{tabular}
    \vspace{0.5cm}
    \caption{Comparing networks I and II for a random system}
    \label{table:random}
\end{table}

\begin{figure*}[h!]

    \centering
    \begin{subfigure}[t]{0.5\textwidth}
        \centering
        \includegraphics[scale=0.3]{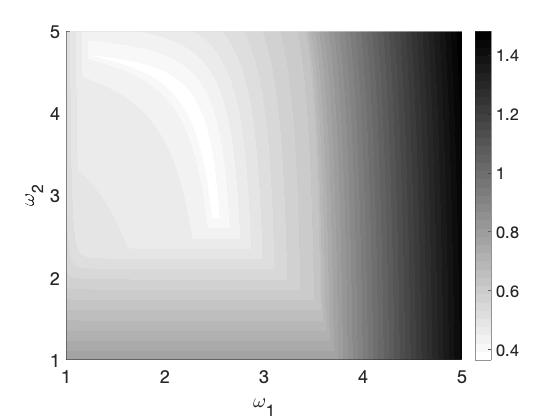}
        \caption{Network I}
    \end{subfigure}%
    ~
    \begin{subfigure}[t]{0.5\textwidth}
        \centering
        \includegraphics[scale=0.3]{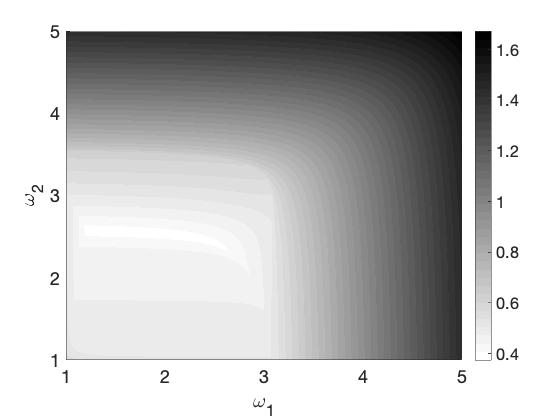}
        \caption{Network II}
    \end{subfigure}
    \caption{Spectral radii for a nearly orthogonal system}
    \label{fig:specradi_orth}
\end{figure*}

For both the nearly orthogonal and random systems, we see that the relaxation parameter is allowed to be larger than 2 to achieve convergence. Note also that the spectral radius $\rho(P^{\vec{\omega}})$ for systems with leaf subnetworks is smaller than the uniform system; hence we see better performance. For the nearly orthogonal systems with leaf subnetworks, we do not need many iterations of the algorithm to achieve a smaller error than the uniform system. However, for the random system, we need many more iterations to achieve this smaller error.

\section{Acknowledgements}
Riley Borgard, Haley Duba, Chloe Makdad, Jay Mayfield, and Randal Tuggle were supported by the National Science Foundation through the REU award \#1457443. 
Steven Harding and Eric Weber were supported by the National Science Foundation and the National Geospatial-Intelligence Agency under award \#1830254.  Eric Weber was also supported under award \#1934884.

\bibliographystyle{amsplain}
\bibliography{bibliography.bib,kaczmarz-reu.bib}

\end{document}